\documentclass[11pt]{llncs}

\usepackage{url}
\usepackage{amssymb}
\usepackage{amscd}
\usepackage{amsfonts}
\usepackage{amsmath}
\usepackage{eepic}
\usepackage{gastex}
\usepackage{amsgen}
\usepackage{eucal}
\usepackage[square]{natbib}

\usepackage{tikz}
\usetikzlibrary{arrows}
\usetikzlibrary{arrows,shapes}
\usetikzlibrary{tqft}

\newcommand{\Rms}{Rees matrix \sgp}
\newcommand{\Rmss}{Rees matrix \sgps}
\newcommand{\sgp}{semi\-group}
\newcommand{\sgps}{semi\-groups}

\DeclareMathOperator{\al}{alph}
\DeclareMathOperator{\occ}{occ}

\newcommand{\wire}[2]{\begin{picture}(13,3)\gasset{AHnb=0,linewidth=0.3}\drawline(3,1)(8,1)\put(0,0){$#1$}\put(9,0){$#2$}\end{picture}}
\newcommand{\wirei}[2]{\begin{picture}(13,3)\gasset{AHnb=0,linewidth=0.3}\drawline(4,1)(8,1)\put(0,0){$#1$}\put(9,0){$#2$}\end{picture}}
\newcommand{\wireii}[2]{\begin{picture}(16,3)\gasset{AHnb=0,linewidth=0.3}\drawline(4,1)(8,1)\put(0,0){$#1$}\put(9,0){$#2$}\end{picture}}
\newcommand{\wires}[2]{\begin{picture}(12,3)\gasset{AHnb=0,linewidth=0.3}\drawline(3,1)(8,1)\put(0,0){$#1$}\put(9,0){$#2$}\end{picture}}

\title{Identities of the Kauffman Monoid $\mathcal{K}_4$\\ and of the Jones monoid $\mathcal{J}_4$}

\titlerunning{Identities of $\mathcal{K}_4$ and $\mathcal{J}_4$}

\author{N. V. Kitov \and M. V. Volkov\thanks{Supported by the Ministry of Science and Higher Education of the Russian Federation, project no.\ 1.580.2016, and the Competitiveness Enhancement Program of Ural Federal University.}}

\authorrunning{N. V. Kitov, M. V. Volkov}

\tocauthor{N. V. Kitov, M. V. Volkov (Ekaterinburg, Russia)}

\institute{Institute of Natural Sciences and Mathematics\\ Ural Federal University, Lenina 51, 620000 Ekaterinburg, Russia\\
\email{n.v.kitov@urfu.ru, m.v.volkov@urfu.ru}}

\begin{document}

\maketitle

\begin{abstract}
Kauffman monoids $\mathcal{K}_n$ and Jones monoids $\mathcal{J}_n$, $n=2,3,\dots$, are two families of monoids relevant in knot theory. We prove a somewhat counterintuitive result that the Kauffman monoids $\mathcal{K}_3$ and $\mathcal{K}_4$ satisfy exactly the same identities. This leads to a polynomial time algorithm to check whether a given identity holds in $\mathcal{K}_4$. As a byproduct, we also find a polynomial time algorithm for checking identities in the Jones monoid $\mathcal{J}_4$.
\end{abstract}

\section{Background I: Identities and identity checking}

The present paper deals with the computational complexity of a combinatorial decision problem (identity checking problem) related to certain algebraic structures originated in knot theory (Kauffman and Jones monoids). Since our results and their proofs involve concepts from several different areas, the list of necessary prerequisites is relatively long. We assume the reader's familiarity with basic notions of computational complexity and semigroup theory; see, e.g., the early chapters of \citep{Pa94} and \citep{CP61}, respectively. Modulo these basics, we tried to make the paper self-contained, to a reasonable extent. In particular, in this section we give a quick introduction into semigroup identities and their checking while the next section provides detailed geometric definitions of Kauffman and Jones monoids.

We fix a countably infinite set $X$ which we call an \emph{alphabet} and which elements we refer to as \emph{letters}. The set $X^+$ of finite sequences of letters forms a semigroup under concatenation which is called the \emph{free semigroup over $X$}. Elements of $X^+$ are called \emph{words over $X$}. If $w=x_1\cdots x_\ell$ with $x_1,\dots,x_\ell\in X$ is a word over $X$, the set $\{x_1,\dots,x_\ell\}$ is called the \emph{content} of $w$ and is denoted $\al(w)$ while the number $\ell$ is referred to as the \emph{length} of $w$ and is denoted $|w|$. We say that a letter $x\in X$ \emph{occurs} in a word $w\in X^+$ or, alternatively, $w$ \emph{involves} $x$ whenever $x\in\al(w)$.

An \emph{identity} is an expression of the form $w\bumpeq w'$ with $w,w'\in X^+$. If $\mathcal{S}$ is a semigroup, we say that the identity $w\bumpeq w'$ \emph{holds} in $\mathcal{S}$ or, alternatively, $\mathcal{S}$ \emph{satisfies} $w\bumpeq w'$ if $w\varphi=w'\varphi$ for every homomorphism $\varphi\colon X^+\to\mathcal{S}$. If $w\bumpeq w'$ does not holds in $\mathcal{S}$, we say that it \emph{fails} in $\mathcal{S}$.

The following observations are immediate: if a semigroup $\mathcal{S}$ satisfies an identity $w\bumpeq w'$, so do each subsemigroup and each quotient of $\mathcal{S}$; if semigroups $\mathcal{S}_1$  and $\mathcal{S}_2$ satisfy $w\bumpeq w'$, so does their direct product $\mathcal{S}_1\times\mathcal{S}_2$.

It is well known and easy to see that the free semigroup $X^+$ possesses the following universal property: for every semigroup $\mathcal{S}$, every mapping $X\to\mathcal{S}$ uniquely extends to a homomorphism $X^+\to\mathcal{S}$. Thus, the homomorphisms $X^+\to\mathcal{S}$  are in a 1-1 correspondence with the mappings $X\to\mathcal{S}$, which we call \emph{substitutions}. Therefore we can restate the fact of $w\bumpeq w'$ holding in $\mathcal{S}$ also in the following terms: every substitution of elements in $\mathcal{S}$ for letters in $X$ yields equal values to $w$ and $w'$.

Given a semigroup $\mathcal{S}$, its \emph{identity checking problem}\footnote{Also called the `\emph{term equivalence problem}' in the literature.} \textsc{Check-Id}($\mathcal{S}$) is the following decision problem. The instance of \textsc{Check-Id}($\mathcal{S}$) is an arbitrary identity $w\bumpeq w'$. The answer to the instance $w\bumpeq w'$ is `YES' whenever the identity $w\bumpeq w'$ holds in $\mathcal{S}$; otherwise, the answer is `NO'.

We stress that here $\mathcal{S}$ is fixed and it is the identity $w\bumpeq w'$ that serves as the input so that the time/space complexity of \textsc{Check-Id}($\mathcal{S}$) should be measured in terms of the size of the identity, that is, in $|ww'|$.

Studying computational complexity of identity checking in semigroups (and other `classical' algebras such as groups and rings) was proposed by Sapir in the influential survey \citep{KS95}, see Problem~2.4 therein. For a \textbf{finite} semigroup $\mathcal{S}$, the problem \textsc{Check-Id}($\mathcal{S}$) is always decidable. Indeed, given an identity $w\bumpeq w'$, there are only finitely many substitutions of elements in $\mathcal{S}$ for letters in $\al(ww')$, and one can check whether or not each of these substitutions yields equal values to $w$ and $w'$. Moreover, \textsc{Check-Id}($\mathcal{S}$) with $\mathcal{S}$ being finite belongs to the complexity class $\mathsf{coNP}$: if for some words $w,w'$ that involve $m$ letters in total, the identity $w\bumpeq w'$ fails in the semigroup $\mathcal{S}$, then a nondeterministic algorithm can guess an $m$-tuple of elements in $\mathcal{S}$ witnessing the failure and then verify the guess by computing the values of the words $w$ and $w'$ under the substitution that sends the letters occurring in $w\bumpeq w'$ to the entries of the guessed $m$-tuple. With multiplication in $\mathcal{S}$ assumed to be performed in unit time, the algorithm takes linear in $|ww'|$ time.

In the literature, there exists many examples of finite semigroups whose identity checking problem is $\mathsf{coNP}$-complete; see, e.g., \citep{AVG09,HLMS,JM06,Ki04,Kl09,Kl12,PV06,Se05,SS06} and the references therein. However, the task of classifying finite semigroups according to the computational complexity of identity checking appears to be far from being feasible. In particular, it is not yet accomplished even in the case when a semigroup under consideration is a finite group. Just to give a hint of difficulties that one encounters when approaching this task, we mention the following result by \cite{Kl09}: a finite semigroup $\mathcal{S}$ with \textsc{Check-Id}($\mathcal{S}$) in $\mathsf{P}$ may have both a subsemigroup and a quotient whose identity checking problems are $\mathsf{coNP}$-complete.

Studying the identity checking problem for \textbf{infinite} semigroups cannot rely on the `finite' methods outlined above. Clearly, the brute-force approach of checking through all possible substitutions fails since the set of such substitutions becomes infinite if their range is an infinite semigroup. The nondeterministic guessing algorithm also fails in general because an infinite semigroup $\mathcal{S}$ may have undecidable word problem
so that it might be impossible to decide whether or not the values of two words under a substitution are equal in $\mathcal{S}$. \citet{Mu68} had constructed an infinite semigroup $\mathcal{S}$ such that the problem \textsc{Check-Id}($\mathcal{S}$) is undecidable. On the other hand, for many `natural' infinite semigroups such as \sgps\ of transformations of an infinite set, or \sgps\ of relations on an infinite domain, or \sgps\ of matrices over an infinite ring, the identity checking problem trivializes since such `big' \sgps\ satisfy only \emph{trivial} identities, that is, identities of the form $w\bumpeq w$. Yet another class of `natural' infinite semigroups with easy identity checking is formed  by various commutative structures in arithmetics and algebra such as integer numbers or real polynomials, say, under addition or multiplication. It is folklore that these commutative semigroups satisfy exactly so-called balanced identities. (An identity $w\bumpeq w'$ is said to be \emph{balanced} if every letter occurs in $w$ and $w'$ the same number of times.  Clearly, this condition can be verified in linear in $|ww'|$ time.)

For a long time, there were no results on the computational complexity of identity checking for infinite semigroups, except for the two aforementioned extremes---undecidability and trivial or easy decidability in linear time.
Only recently, the situation has started to change, and a few examples of infinite semigroups with identity checking decidable in a nontrivial way have appeared. An interesting instance here is the so-called bicyclic monoid $\mathcal{B}$ generated by two elements $a$ and $b$ subject to the relation $ba=1$; this monoid is known to play a distinguished role in the structure theory of semigroups. The fact that $\mathcal{B}$ satisfies a nontrivial identity was first discovered by \cite{Adian:1962}. After that, various combinatorial, computational, and geometric aspects of identities holding in $\mathcal{B}$ were examined in the literature, see, e.g., \citep{Sh89,Shl90,Pa06}, but only short while ago \cite{DJK18} have shown that checking identities in $\mathcal{B}$ can be done in polynomial time via quite a tricky algorithm based on linear programming. Another example is the Kauffman monoid $\mathcal{K}_3$ generated by three elements $h_1$, $h_2$, and $c$ subject to the relations $h_ih_{3-i}h_i=h_i$ and $h_i^2=ch_i=h_ic$,  $i=1,2$; a recent paper by \cite{Chen20} provides an algorithm for checking identities in $\mathcal{K}_3$ in quasilinear time. The main result of the present paper extends this algorithm to the Kauffman monoid $\mathcal{K}_4$, which we define next.

\section{Background II: Kauffman and Jones monoids}
\label{sec:k&j}

Let $n$ be an integer greater than 1. The \emph{Kauffman monoid}\footnote{The name comes from \citep{BDP02}; in the literature one also meets the name \emph{Temperley--Lieb--Kauffman monoids} \citep[see, e.g.,][]{BL05}.}  $\mathcal{K}_n$ can be defined as the monoid with $n$ generators $c,h_1,\dots,h_{n-1}$ subject to the following relations:
\begin{align}
&h_{i}h_{j}=h_{j}h_{i}    &&\text{if } |i-j|\ge 2,\ i,j=1,\dots,n-1;\label{eq:TL1}\\
&h_{i}h_{j}h_{i}=h_{i}    &&\text{if } |i-j|=1,\ i,j=1,\dots,n-1;\label{eq:TL2}\\
&h_{i}^2=ch_{i}=h_{i}c    &&\text{for each } i=1,\dots,n-1.\label{eq:TL3}
\end{align}
Kauffman monoids play an important role in knot theory, low-dimensional topology, topological quantum field theory, quantum groups, etc. As algebraic objects, these monoids belong to the family of so-called diagram or Brauer-type monoids that originally arose in representation theory \citep{Br37} and have been intensively studied from various viewpoints over the last two decades; see, e.g., \citep{Au12,Au14,ADV12,ACHLV15,DE17,DE18,Dea15,DEG17,Dea19,Ea11a,Ea11b,Ea14a,Ea14b,Ea18,Ea19a,Ea19b,EF12,EG17,Eea18,FL11,KMM06,KM06,KM07,LF06,MM07,Ma98,Ma02} and references therein.

It is convenient to use, along with the above definition of the monoids $\mathcal{K}_n$ in terms of generators and relations, their more geometric definition due to \citet{Ka90}. We present the latter definition, following~\citep{ACHLV15}, where the monoids $\mathcal{K}_n$ arise as `planar' submonoids in monoids from a more general (but easier to define) family.

Let $[n]:=\{1,\dots,n\}$, $[n]':=\{1',\dots,n'\}$ be two disjoint copies of the set of the first $n$ positive integers. Consider the set $\mathcal{W}_n$ of all pairs $(\pi;s)$ where $\pi$ is a partition of the $2n$-element set $[n]\cup [n]'$ into 2-element blocks and $s$ is a nonnegative integer referred to as the \emph{number of circles}. Such a pair is represented by a \emph{wire diagram} as shown in Fig.~\ref{fig:diagram}.
\begin{figure}[ht]
\centering
\unitlength=.7mm
\begin{picture}(45,65)(0,15)
\multiput(4,17)(0,8){9}{$\bullet$}
\multiput(39,17)(0,8){9}{$\bullet$}
\gasset{AHnb=0,linewidth=.5}
\drawline(5,18)(40,50)
\drawline(5,58)(40,82)
\drawline(5,74)(40,74)
\drawcurve(5,26)(15,33)(5,42)
\drawcurve(5,34)(15,42)(5,50)
\drawcurve(5,66)(15,74)(5,82)
\drawcurve(40,18)(35,22)(40,26)
\drawcurve(40,34)(35,38)(40,42)
\drawcurve(40,58)(35,62)(40,66)
\put(0,17){$\scriptstyle 1$}
\put(0,25){$\scriptstyle 2$}
\put(0,33){$\scriptstyle 3$}
\put(0,41){$\scriptstyle 4$}
\put(0,49){$\scriptstyle 5$}
\put(0,57){$\scriptstyle 6$}
\put(0,65){$\scriptstyle 7$}
\put(0,73){$\scriptstyle 8$}
\put(0,81){$\scriptstyle 9$}
\put(43,17){$\scriptstyle 1'$}
\put(43,25){$\scriptstyle 2'$}
\put(43,33){$\scriptstyle 3'$}
\put(43,41){$\scriptstyle 4'$}
\put(43,49){$\scriptstyle 5'$}
\put(43,57){$\scriptstyle 6'$}
\put(43,65){$\scriptstyle 7'$}
\put(43,73){$\scriptstyle 8'$}
\put(43,81){$\scriptstyle 9'$}
\drawcircle(22.5,50,3)
\drawcircle(22.5,50,6)
\drawcircle(22.5,50,9)
\end{picture}
\caption{Wire diagram representing an element of $\mathcal{W}_9$}\label{fig:diagram}
\end{figure}
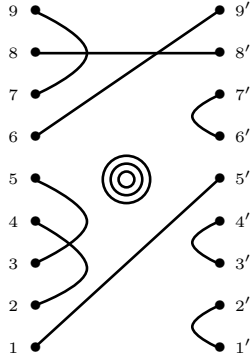
We represent the elements of $[n]$ by points on the left hand side of the diagram (\emph{left points}) while the elements of $[n]'$ are represented by points on the right hand side of the diagram (\emph{right points}). We will omit the labels $1,2,\dots,1',2',\dots$ in our further illustrations. Now, for $(\pi;s)\in \mathcal{W}_n$, we represent the number $s$ by $s$ closed curves (`circles') drawn somewhere within the diagram and each block of the partition $\pi$ is represented by a line referred to as a \emph{wire}. Thus, each wire connects two points; it is called an $\ell$-\emph{wire} if it connects two left points, an $r$-\emph{wire} if it connects two right points,
and a $t$-\emph{wire} if it connects a left point with a right point. The wire diagram in Fig.~\ref{fig:diagram} has three wires of each type and corresponds to the pair
$$\Bigl(\bigr\{\{1,5'\},\{2,4\},\{3,5\},\{6,9'\},\{7,9\},\{8,8'\},\{1',2'\},\{3',4'\},\{6',7'\}\bigr\};\,3\Bigr).$$

Now we define a multiplication in $\mathcal{W}_n$. Pictorially, in order to multiply two diagrams, we glue their wires together by identifying each right point $u'$ of the first diagram with the corresponding left point $u$ of the second diagram. This way we obtain a new diagram whose left (respectively, right) points are the left (respectively, right) points of the first (respectively, second) diagram. Two points of this new diagram are connected in it if one can reach one of them from the other by walking along a sequence of consecutive wires of the factors, see Fig.~\ref{fig:multiplication}. All circles of the factors are inherited by the product; in addition, some extra circles may arise from $r$-wires of the first diagram combined with $\ell$-wires of the second diagram.

In more precise terms, if $\xi=(\pi_1;s_1)$, $\eta=(\pi_2;s_2)$, then a left point $p$ and a right point $q'$ of the product $\xi\eta$ are connected by a $t$-wire if and only if one of the following conditions holds:
\begin{trivlist}
\item[$\bullet$] \wire{p}{u'} is a $t$-wire in $\xi$ and \wire{u}{q'} is a $t$-wire in $\eta$ for some $u\in[n]$;
\item[$\bullet$] for some $s>1$ and some $u_1,v_1,u_2,\dots,v_{s-1},u_s\in[n]$ (all pairwise distinct), \wire{p}{u_1'} is a $t$-wire in $\xi$ and \wirei{u_s}{q'} is a $t$-wire in $\eta$, while \wirei{u_i}{v_i} is an $\ell$-wire in $\eta$ and \wireii{v_i'}{u_{i+1}'} is an $r$-wire in $\xi$ for each $i=1,\dots,s-1$.\\
(The reader may trace an application of the second rule in Fig.~\ref{fig:multiplication}, in which such a `composite' $t$-wire connects 1 and $3'$ in the product diagram.)
\end{trivlist}
\begin{figure}[ht]
\centering
\unitlength=.67mm
\begin{picture}(170,65)(0,15)
\gasset{AHnb=0,linewidth=.5}
\multiput(4,17)(0,8){9}{$\bullet$}
\multiput(39,17)(0,8){9}{$\bullet$}
\drawline(5,18)(40,50)
\drawline(5,58)(40,82)
\drawline(5,74)(40,74)
\drawcurve(5,26)(15,33)(5,42)
\drawcurve(5,34)(15,42)(5,50)
\drawcurve(5,66)(15,74)(5,82)
\drawcurve(40,18)(35,22)(40,26)
\drawcurve(40,34)(35,38)(40,42)
\drawcurve(40,58)(35,62)(40,66)
\drawcircle(22.5,50,3)
\drawcircle(22.5,50,6)
\drawcircle(22.5,50,9)
\put(43.5,48){$\times$}
\put(46,0){\begin{picture}(50,80)
\gasset{AHnb=0,linewidth=.5,linecolor=red}
\multiput(4,17)(0,8){9}{$\bullet$}
\multiput(39,17)(0,8){9}{$\bullet$}
\drawline(5,66)(40,34)
\drawcurve(5,18)(10,22)(5,26)
\drawcurve(5,34)(15,46)(5,58)
\drawcurve(5,42)(10,46)(5,50)
\drawcurve(5,74)(10,78)(5,82)
\drawcurve(40,18)(35,22)(40,26)
\drawcurve(40,50)(35,54)(40,58)
\drawcurve(40,66)(35,70)(40,74)
\drawcurve(40,42)(30,62)(40,82)
\drawcircle(22.5,35,4)
\end{picture}}
\put(90.5,48){=}
\put(94,0){\begin{picture}(80,80)
\multiput(4,17)(0,8){9}{$\bullet$}
\multiput(74,17)(0,8){9}{$\bullet$}
\drawline(5,18)(40,50)
\drawline(5,58)(40,82)
\drawline(5,74)(40,74)
\drawcurve(5,26)(15,33)(5,42)
\drawcurve(5,34)(15,42)(5,50)
\drawcurve(5,66)(15,74)(5,82)
\drawcurve(40,18)(35,22)(40,26)
\drawcurve(40,34)(35,38)(40,42)
\drawcurve(40,58)(35,62)(40,66)
\drawcircle(22.5,50,3)
\drawcircle(22.5,50,6)
\drawcircle(22.5,50,9)
\drawline[linewidth=.3,dash={2.05}0](40,12)(40,88)
\put(35,0){\begin{picture}(40,80)
\gasset{AHnb=0,linewidth=.5,linecolor=red}
\drawline(5,66)(40,34)
\drawcurve(5,18)(10,22)(5,26)
\drawcurve(5,34)(15,46)(5,58)
\drawcurve(5,42)(10,46)(5,50)
\drawcurve(5,74)(10,78)(5,82)
\drawcurve(40,18)(35,22)(40,26)
\drawcurve(40,50)(35,54)(40,58)
\drawcurve(40,66)(35,70)(40,74)
\drawcurve(40,42)(30,62)(40,82)
\drawcircle(22.5,35,4)
\end{picture}}
\end{picture}}
\end{picture}
\caption{Multiplication of wire diagrams}\label{fig:multiplication}
\end{figure}
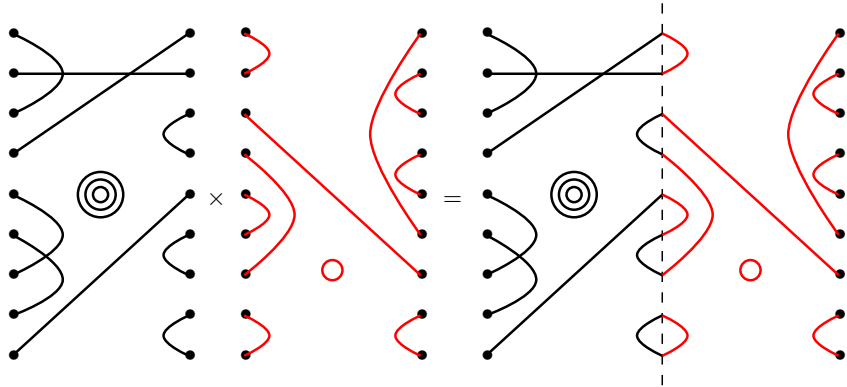

Analogous characterizations hold for the $\ell$-wires and $r$-wires of $\xi\eta$. Here we include only the rules for forming $\ell$-wires as the $r$-wires of the product are obtained in a perfectly symmetric way.

Two left points $p$ and $q$ of $\xi\eta$ are connected by an $\ell$-wire if and only if one of the following conditions holds:
\begin{trivlist}
\item[$\bullet$] \wires{p}{q} is an $\ell$-wire in $\xi$;
\item[$\bullet$] for some $s\ge1$ and some $u_1,v_1,u_2,\dots,v_s\in[n]$ (all pairwise distinct), \wire{p}{u_1'} and \wire{q}{v_s'} are $t$-wires in $\xi$, while \wirei{u_i}{v_i} is an $\ell$-wire in $\eta$ for each $i=1,\dots,s$ and if $s>1$, then \wireii{v_i'}{u_{i+1}'} is an $r$-wire in $\xi$ for each $i=1,\dots,s-1$.\\
(Again, Fig.~\ref{fig:multiplication} provides an instance of the second rule: look at the $\ell$-wire that connects 6 and 8 in the product diagram.)
\end{trivlist}

Finally, each circle of the product $\xi\eta$ corresponds to either a circle in $\xi$ or $\eta$ or a sequence $u_1,v_1,\dots,u_s,v_s\in[n]$ with $s\ge 1$ and pairwise distinct $u_1,v_1,\dots,u_s,v_s$ such that all \wirei{u_i}{v_i} are $\ell$-wires in $\eta$, while all \wireii{v_i'}{u_{i+1}'} and \wirei{v_s'}{u_1'} are $r$-wires in $\xi$.

It easy to see that the above defined multiplication in $\mathcal{W}_n$ is associative and that the diagram with 0 circles and the $n$ horizontal $t$-wires \wires{1}{1'}, \dots, \wire{n}{n'} is the identity element with respect to the multiplication. Thus, $\mathcal{W}_n$ is a monoid that we term the \emph{wire monoid}.

\citet{Ka90} has defined the \emph{connection monoid} $\mathcal{C}_n$ as the submonoid of $\mathcal{W}_n$ consisting of all elements of $\mathcal{W}_n$ that have a representation as a diagram whose wires do not cross. (Thus, the left factor and the product in the multiplication example in Fig.~\ref{fig:multiplication} are not elements of $\mathcal{C}_n$, while the right factor lies in $\mathcal{C}_n$.) Kauffman has shown that $\mathcal{C}_n$ is generated by the \emph{hooks} $h_1,\dots,h_{n-1}$, where
$$h_i:=\Bigl(\bigr\{\{i,i+1\},\{i',(i+1)'\},\{j,j'\}\mid \text{for all } j\ne i,i+1\bigr\};\,0\Bigr),$$
and the circle $c:=\Bigl(\bigr\{\{j,j'\}\mid \text{for all } j=1,\dots,n\bigr\};\,1\Bigr),$ see Fig.~\ref{fig:hooks} for an illustration.
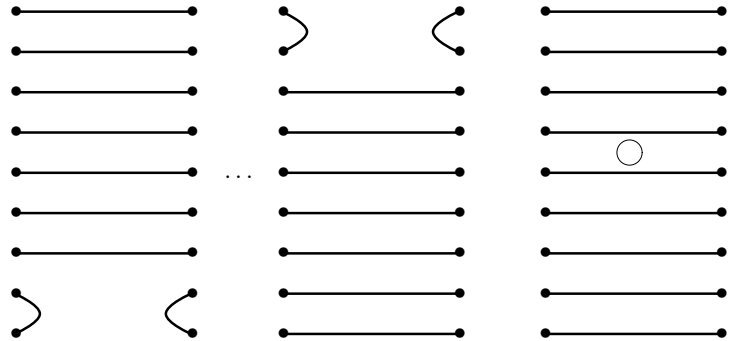
\begin{figure}[ht]
\centering \unitlength=.67mm
\begin{picture}(150,65)(0,15)
\gasset{AHnb=0,linewidth=.5}
\multiput(4,17)(0,8){9}{$\bullet$}
\multiput(39,17)(0,8){9}{$\bullet$}
\drawcurve(5,18)(10,22)(5,26)
\drawline(5,34)(40,34)
\drawline(5,42)(40,42)
\drawline(5,50)(40,50)
\drawline(5,58)(40,58)
\drawline(5,66)(40,66)
\drawline(5,74)(40,74)
\drawline(5,82)(40,82)
\drawcurve(40,18)(35,22)(40,26)
\put(46.5,49){$\dots$}
\put(53,0){\begin{picture}(50,80)
\multiput(4,17)(0,8){9}{$\bullet$}
\multiput(39,17)(0,8){9}{$\bullet$}
\drawline(5,18)(40,18)
\drawline(5,26)(40,26)
\drawline(5,34)(40,34)
\drawline(5,42)(40,42)
\drawline(5,50)(40,50)
\drawline(5,58)(40,58)
\drawline(5,66)(40,66)
\drawcurve(5,74)(10,78)(5,82)
\drawcurve(40,74)(35,78)(40,82)
\end{picture}}
\put(105,0){\begin{picture}(50,80)
\multiput(4,17)(0,8){9}{$\bullet$}
\multiput(39,17)(0,8){9}{$\bullet$}
\drawline(5,18)(40,18)
\drawline(5,26)(40,26)
\drawline(5,34)(40,34)
\drawline(5,42)(40,42)
\drawline(5,50)(40,50)
\drawline(5,58)(40,58)
\drawline(5,66)(40,66)
\drawline(5,74)(40,74)
\drawline(40,82)(5,82)
\put(22,54){\circle{5}}
\end{picture}}
\end{picture}
\caption{The hooks $h_1,\dots,h_8$ and the circle $c$ in $\mathcal{C}_9$}\label{fig:hooks}
\end{figure}
It is easy to check that the generators $h_1,\dots,h_{n-1},c$ satisfy the relations \eqref{eq:TL1}--\eqref{eq:TL3}, whence there exists a homomorphism from the Kauffman monoid $\mathcal{K}_n$ onto the connection monoid $\mathcal{C}_n$. In fact, this homomorphism is an isomorphism between $\mathcal{K}_n$ and $\mathcal{C}_n$; see \citep{Ka90} for a proof outline and \citep{BDP02} for a very detailed argument. Thus, we may (and will)
identify $\mathcal{K}_n$ with $\mathcal{C}_n$ in what follows.

Denote by $\mathcal{J}_n$ the set of all diagrams in $\mathcal{K}_n$ without circles. Observe that this set is finite; in fact, it is known that the cardinality of $\mathcal{J}_n$ is the $n$-th Catalan number $\dfrac{1}{n+1}\dbinom{2n}{n}$. We define the multiplication of two diagrams in $\mathcal{J}_n$ as follows: we multiply the diagrams as elements of $\mathcal{K}_n$ and then reduce the product to a diagram in $\mathcal{J}_n$ by removing all circles. This multiplication makes $\mathcal{J}_n$ a monoid known as the \emph{Jones monoid}\footnote{The name was suggested by \citet{LF06} to honor the contribution of V.F.R.~Jones to the theory
\citep[see, e.g.,][Section~4]{Jo83}.}. Observe that $\mathcal{J}_n$ is \textbf{not} a submonoid of $\mathcal{K}_n$; at the same time, the `erasing' map $\xi\mapsto\bar{\xi}$ that forgets the circles of each diagram $\xi\in\mathcal{K}_n$ is easily seen to be a surjective homomorphism of $\mathcal{K}_n$ onto $\mathcal{J}_n$. The hooks $h_1,\dots,h_{n-1}$ clearly satisfy $\bar{h}_i=h_i$ while $\bar{c}$ is the identity element of $\mathcal{J}_n$. This implies that the monoid $\mathcal{J}_n$ is generated by $\bar{h}_1,\dots,\bar{h}_{n-1}$ and that $\bar{h}_{i}^2=\bar{h}_{i}$ for each $i=1,\dots,n-1$. Moreover, if $\|\xi\|$ stands for the number of circles of the diagram $\xi\in\mathcal{K}_n$, then the map $\xi\mapsto\bigl(\bar{\xi},\|\xi\|\bigr)$ is a bijection between $\mathcal{K}_n$ and the cartesian product of $\mathcal{J}_n$ with the set $\mathbb{N}_0$ of nonnegative integers. Here is a simple formula for multiplying diagrams from $\mathcal{K}_n$ in these `coordinates':
\begin{equation}
\label{eq:coordinates}
\bigl(\bar{\xi},\|\xi\|\bigr)\cdot\bigl(\bar{\eta},\|\eta\|\bigr)=\bigl(\bar{\xi}\bar{\eta},\|\xi\|+\|\eta\|+\langle\bar{\xi},\bar{\eta}\rangle\bigr),
\end{equation}
where the term $\langle\bar{\xi},\bar{\eta}\rangle$ denotes the number of circles removed when the product $\bar{\xi}\bar{\eta}$ in $\mathcal{J}_n$ is formed.

Now, following an idea by Auinger (personal communication), we embed the monoid $\mathcal{K}_n$ into a larger monoid $\widehat{\mathcal{K}}_n$ which is easier to deal with. In terms of generators and relations,
the \emph{extended Kauffman monoid} $\widehat{\mathcal{K}}_n$ can be defined as the monoid with $n+1$ generators $c,d,h_1,\dots,h_{n-1}$ subject to the relations \eqref{eq:TL1}--\eqref{eq:TL3} and the additional relations
\begin{equation}
\label{eq:inverse}
cd=dc=1.
\end{equation}
Observe that the relations \eqref{eq:TL3} and \eqref{eq:inverse} imply that $dh_i=h_id$ for each $i=1,\dots,n-1$. Indeed,
\begin{align*}
dh_i&=d^2ch_i&&\text{since $dc=1$}\\
    &=d^2h_ic&&\text{since $ch_i=h_ic$}\\
    &=d^2h_ic^2d&&\text{since $cd=1$}\\
    &=d^2c^2h_id&&\text{since $c^2h_i=h_ic^2$}\\
    &=h_id&&\text{since $d^2c^2=1$.}
\end{align*}
It is easy to see that the submonoid of $\widehat{\mathcal{K}}_n$ generated by $c,h_1,\dots,h_{n-1}$ is isomorphic to $\mathcal{K}_n$.

The interpretation of the extended Kauffman monoid in terms of diagrams is a bit less natural as it requires introducing two sorts of circles: positive and negative. Each diagram may contain only circles of one sort. When two diagrams are multiplied, the following two rules are obeyed: all newly created circles (which arise when the diagrams are glued together) are positive; in addition, if the product diagram inherits some negative circles from its factors, then pairs of `opposite' circles are consecutively removed until only circles of a single sort (or no circles at all) remain. The Kauffman monoid $\mathcal{K}_n$ is then nothing but the submonoid of all diagrams having only positive circles or no circles at all.

Clearly, the `erasing' homomorphism of $\mathcal{K}_n$ onto $\mathcal{J}_n$ extends to the monoid $\widehat{\mathcal{K}}_n$. If we extend also the circle-counting map $\mathcal{K}_n\to\mathbb{N}_0$ to $\widehat{\mathcal{K}}_n$, letting $\|\xi\|=-n$ for each diagram $\xi$ with $n$ negative circles, we get that $\widehat{\mathcal{K}}_n$ can be identified with $\mathcal{J}_n\times\mathbb{Z}$, the cartesian product of the corresponding Jones monoid with the set of all integers, the multiplication on $\mathcal{J}_n\times\mathbb{Z}$ being defined by the formula \eqref{eq:coordinates}.

\section{Rees matrix semigroups and their identities}
\label{sec:rees}

We briefly recall the Rees matrix construction; see \cite[Chapter~3]{CP61} for details and the explanation of the distinguished role played by this construction in the structure theory of semigroups. Let $\mathcal{G}$ be a group, $0$ a symbol beyond $\mathcal{G}$, and $I,\Lambda$ non-empty sets. Given a $\Lambda\times I$-matrix $P=(p_{\lambda i})$ over $\mathcal{G}\cup\{0\}$, we define a multiplication on the set $(I\times\mathcal{G}\times\Lambda)\cup\{0\}$ by the following rules:
\begin{equation}
\begin{gathered}
\label{eq:rms}
a\cdot 0=0\cdot a:=0\ \text{ for all } a\in (I\times\mathcal{G}\times\Lambda)\cup\{0\},\\
(i,g,\lambda)\cdot(j,h,\mu):=\begin{cases}
(i,gp_{\lambda j}h,\mu)&\ \text{if}\ p_{\lambda j}\ne0,\\
0 &\ \text{if}\ p_{\lambda j}=0.
\end{cases}
\end{gathered}
\end{equation}
The multiplication is easily seen to be associative so that $(I\times\mathcal{G}\times\Lambda)\cup\{0\}$ becomes a semigroup. We denoted it by $\mathcal{M}^0(I,\mathcal{G},\Lambda;P)$ and call the \emph{\Rms\ over $\mathcal{G}$ with the sandwich matrix $P$}. If the matrix $P$ has no zero entries, the set $I\times\mathcal{G}\times\Lambda$ forms a subsemigroup in $\mathcal{M}^0(I,\mathcal{G},\Lambda;P)$. We denote this subsemigroup by $\mathcal{M}(I,\mathcal{G},\Lambda;P)$ and apply the name `\Rms' also to it.

We need a combinatorial characterization of identities holding in every \Rms\ over an abelian group. In order to formulate it, we recall a few definitions.

For a semigroup $\mathcal{S}$, the notation $\mathcal{S}^1$ stands for the least monoid containing $\mathcal{S}$, that is, $\mathcal{S}^1:=\mathcal{S}$ if $\mathcal{S}$ has an identity element and $\mathcal{S}^1:=\mathcal{S}\cup\{1\}$ if $\mathcal{S}$ has no identity element. In the latter case the multiplication in $\mathcal{S}$ is extended to $\mathcal{S}^1$ in a unique way such that the fresh symbol $1$ becomes the identity element in $\mathcal{S}^1$. We adopt the following notational convention: for $s\in\mathcal{S}$, the expression $s^0$ stands for the identity element of $\mathcal{S}^1$.

Recall that we have fixed a countably infinite alphabet $X$. The monoid $X^*:=(X^+)^1$ is called the \emph{free monoid} over $X$. We say that a word $v\in X^+$ \emph{occurs} in a word $w\in X^+$ if $w=u_1vu_2$ for some words $u_1,u_2\in X^*$. Clearly, $v$ may have several occurrences in $w$; we denote the number of occurrences of $v$ in $w$ by $\occ_v(w)$.

\begin{proposition}
\label{prop:rms}
An identity $w\bumpeq w'$ holds in every \Rms\ over an abelian group if and only if the words $w$ and $w'$ satisfy the following three conditions:
\begin{itemize}
\item[\emph{(a)}] the first letter of $w$ is the same as the first letter of $w'$;
\item[\emph{(b)}] the last letter of $w$ is the same as the last letter of $w'$;
\item[\emph{(c)}] for each word $v$ of length $2$, $\occ_v(w)=\occ_v(w')$.
\end{itemize}
\end{proposition}

\begin{proof}
The result is basically known. For the special case of \Rmss\ of the form $\mathcal{M}(I,\mathcal{G},\Lambda;P)$, it had been proven by \cite{KR79}; some other special cases were considered in a preprint by \cite{Ma80}. For the reader's convenience, we provide a self-contained proof (which is not difficult at all).

For the `only if' part, let $\mathbb{C}_\infty$ stand for the infinite cyclic group. We fix a generator $c$ for $\mathbb{C}_\infty$ and consider the \Rms\ $\mathcal{S}:=\mathcal{M}\left(\{1,2\},\mathbb{C}_\infty,\{1,2\};P\right)$ where $P:=\begin{pmatrix}e&c\\e&e\end{pmatrix}$, with $e:=c^0$. The identity $w\bumpeq w'$  holds in $\mathcal{S}$. Define a substitution $\alpha\colon X\to\mathcal{S}$ by
\[
x\alpha=\begin{cases}
(1,e,1)&\text{if $x$ is the first letter of $w$,}\\
(2,e,2)&\text{otherwise.}
\end{cases}
\]
By~\eqref{eq:rms}, the first entry of the triple $w\alpha$ is 1, and since $w\alpha=w'\alpha$, so is the first entry of the triple $w'\alpha$. This is only possible provided that $w'$ starts with $x$. We have thus shown that the condition (a) is satisfied. Similarly, by using the substitution $\omega\colon X\to\mathcal{S}$ such that
\[
x\omega=\begin{cases}
(1,e,1)&\text{if $x$ is the last letter of $w$,}\\
(2,e,2)&\text{otherwise,}
\end{cases}
\]
one verifies that (b) holds as well.

In order to verify (c), take a word $v$ of length $2$ that occurs in $w$. First consider the case of $v=yz$, with $y$ and $z$ being distinct. Here we invoke the substitution $\vartheta\colon X\to\mathcal{S}$ such that
\[
x\vartheta=\begin{cases}
(1,e,1)&\text{if $x=y$,}\\
(2,e,2)&\text{if $x=z$,}\\
(1,e,2)&\text{otherwise.}
\end{cases}
\]
Using the rule~\eqref{eq:rms} and the structure of the sandwich matrix $P$, we see that the middle entries of the triples $w\vartheta$ and $w'\vartheta$ are equal to $c^{\occ_{yz}(w)}$ and respectively $c^{\occ_{yz}(w')}$. Since $w\vartheta=w'\vartheta$, we get $\occ_{yz}(w)=\occ_{yz}(w')$.

It remains to analyze the case of $v=y^2$. In this case the substitution $\psi\colon X\to\mathcal{S}$ defined by
\[
x\psi=\begin{cases}
(2,e,1)&\text{if $x=y$,}\\
(1,e,2)&\text{otherwise}
\end{cases}
\]
has the property that the middle entries of the triples $w\psi$ and $w'\phi$ are equal to $c^{\occ_{y^2}(w)}$ and respectively $c^{\occ_{y^2}(w')}$. The equality $w\psi=w'\psi$ yields $\occ_{y^2}(w)=\occ_{y^2}(w')$. Thus,   (c) holds for every word of length $2$.

For the `if' part, we isolate an observation that will be re-used later.
\begin{lemma}
\label{lem:balanced}
If two words $w$ and $w'$ satisfy the conditions \emph{(a)--(c)}, then each letter occurs in $w$ and $w'$ the same number of times.
\end{lemma}

\begin{proof}
For each letter $x\in\al(w)$, we have
\[
\occ_x(w)=\sum_{y\in\al(w)}\occ_{xy}(w)+\begin{cases}1&\text{if the last letter of $w$ is $x$,}\\ 0&\text{otherwise.}\end{cases}
\]
The same formula holds for $w'$ and since, by (c), $\occ_{xy}(w)=\occ_{xy}(w')$ for every letter $y$ and, by (b), $w'$ ends with $x$ if and only if so does $w$, we conclude that $\occ_x(w)=\occ_x(w')$.\qed
\end{proof}

Now consider an arbitrary abelian group $\mathcal{G}$ and an arbitrary \Rms\ $\mathcal{M}^0(I,\mathcal{G},\Lambda;P)$ over $\mathcal{G}$. Take any substitution \[\varphi\colon X\to\mathcal{M}^0(I,\mathcal{G},\Lambda;P).\] If $x\varphi=0$ for some $x\in\al(w)$, then clearly $w\varphi=0$ and, by Lemma~\ref{lem:balanced}, $w'\varphi=0$, too. Thus, assume that $x\varphi\in I\times\mathcal{G}\times\Lambda$ for every $x\in\al(w)$. Let $x\varphi=(i(x),g(x),\lambda(x))$. The multiplication rule~\eqref{eq:rms} then ensures that the equality $w\varphi=0$ is only possible if $p_{\lambda(x)i(y)}=0$ for some (not necessarily distinct) letters $x,y$ such that the word $xy$ occurs in $w$. By (c), $xy$ occurs also in $w'$ whence $w'\varphi=0$. By symmetry, $w'\varphi=0$ implies $w\varphi=0$.

It remains to analyze the situation with both $w\varphi\ne0$ and $w'\varphi\ne0$, in which case $p_{\lambda(x)i(y)}\in\mathcal{G}$ whenever the word $xy$ occurs in $w$. Let $x_{\mathrm{first}}$ and $x_{\mathrm{last}}$ be the first and respectively the last letter of $w$. Using the rule~\eqref{eq:rms} and the fact that the group $\mathcal{G}$ is abelian, one readily computes that $w\varphi=(i(x_{\mathrm{first}}),g,\lambda(x_{\mathrm{last}}))$, with the middle entry $g$ given by the following expression:
\[
g=\prod_{x\in\al(w)}g(x)^{\occ_x(w)}\quad\times\prod_{\substack{x,y\in\al(w)\\xy\text{ occurs in }w}}p_{\lambda(x)i(y)}^{\occ_{xy}(w)}.
\]
In view of (a)--(c) and Lemma~\ref{lem:balanced}, we get $w'\varphi=(i(x_{\mathrm{first}}),g,\lambda(x_{\mathrm{last}}))$, with the same group entry $g$. Hence, the equality $w\varphi=w'\varphi$ holds.\qed
\end{proof}

\section{Structure and identities of $\mathcal{J}_4$}
\label{sec:j4}

The main aim of the present paper is the identity checking problem for the Kauffman monoid $\mathcal{K}_4$. In view of the bijection between $\mathcal{K}_4$ and $\mathcal{J}_4\times\mathbb{N}_0$, it is handy to have a closer look at the Jones monoid $\mathcal{J}_4$. The latter monoid consists of $\dfrac{1}{5}\dbinom{8}{4}=14$ diagrams: the identity diagram with four $t$-wires, nine diagrams with two $t$-wires, and four diagrams without $t$-wires. Fig.~\ref{fig:jones} shows the nonidentity diagrams in $\mathcal{J}_4$.

\begin{figure}[htp]
\begin{center}
\begin{tikzpicture}
[scale=0.55]
\foreach \x in {0,2.5,5,7.5,10,12.5} \foreach \y in {0,1,2,3,5,6,7,8,10,11,12,13} \filldraw (\x,\y) circle (4pt);
\node[] at (1.4,9.4) {$\bar{h}_3$};
\draw (0,10) -- (2.5,10);
\draw (0,11) -- (2.5,11);
\draw (0,12) .. controls (0.5,12.5) .. (0,13);
\draw (2.5,12) .. controls (2,12.5) .. (2.5,13);
\node[] at (6.4,9.4) {$\bar{h}_3\bar{h}_2$};
\draw (5,10) -- (7.5,10);
\draw (5,11) -- (7.5,13);
\draw (5,12) .. controls (5.5,12.5) .. (5,13);
\draw (7.5,11) .. controls (7,11.5) .. (7.5,12);
\node[] at (11.4,9.4) {$\bar{h}_3\bar{h}_2\bar{h}_1$};
\draw (10,10) -- (12.5,12);
\draw (10,11) -- (12.5,13);
\draw (10,12) .. controls (10.5,12.5) .. (10,13);
\draw (12.5,10) .. controls (12,10.5) .. (12.5,11);
\node[] at (1.4,4.4) {$\bar{h}_2\bar{h}_3$};
\draw (0,5) -- (2.5,5);
\draw (0,6) .. controls (0.5,6.5) .. (0,7);
\draw (0,8) -- (2.5,6);
\draw (2.5,7) .. controls (2,7.5) .. (2.5,8);
\node[] at (6.4,4.4) {$\bar{h}_2$};
\draw (5,5) -- (7.5,5);
\draw (5,6) .. controls (5.5,6.5) .. (5,7);
\draw (7.5,6) .. controls (7,6.5) .. (7.5,7);
\draw (5,8) -- (7.5,8);
\node[] at (11.4,4.4) {$\bar{h}_2\bar{h}_1$};
\draw (10,5) -- (12.5,7);
\draw (10,6) .. controls (10.5,6.5) .. (10,7);
\draw (10,8) -- (12.5,8);
\draw (12.5,5) .. controls (12,5.5) .. (12.5,6);
\node[] at (1.4,-0.6) {$\bar{h}_1\bar{h}_2\bar{h}_3$};
\draw (0,0) .. controls (0.5,0.5) .. (0,1);
\draw (0,2) -- (2.5,0);
\draw (0,3) -- (2.5,1);
\draw (2.5,2) .. controls (2,2.5) .. (2.5,3);
\node[] at (6.4,-0.6) {$\bar{h}_1\bar{h}_2$};
\draw (5,0) .. controls (5.5,0.5) .. (5,1);
\draw (5,2) -- (7.5,0);
\draw (5,3) -- (7.5,3);
\draw (7.5,1) .. controls (7,1.5) .. (7.5,2);
\node[] at (11.4,-0.6) {$\bar{h}_1$};
\draw (10,0) .. controls (10.5,0.5) .. (10,1);
\draw (12.5,0) .. controls (12,0.5) .. (12.5,1);
\draw (10,2) -- (12.5,2);
\draw (10,3) -- (12.5,3);
\end{tikzpicture}

\vskip .55cm

\begin{tikzpicture}
[scale=0.55]
\foreach \x in {5,7.5,10,12.5} \foreach \y in {0,1,2,3,5,6,7,8} \filldraw (\x,\y) circle (4pt);
\node[] at (6.4,4.4) {$\bar{h}_1\bar{h}_3$};
\draw (5,5) .. controls (5.5,5.5) .. (5,6);
\draw (7.5,5) .. controls (7,5.5) .. (7.5,6);
\draw (5,7) .. controls (5.5,7.5) .. (5,8);
\draw (7.5,7) .. controls (7,7.5) .. (7.5,8);
\node[] at (11.4,4.4) {$\bar{h}_1\bar{h}_3\bar{h}_2$};
\draw (10,5) .. controls (10.5,5.5) .. (10,6);
\draw (10,7) .. controls (10.5,7.5) .. (10,8);
\draw (12.5,5) .. controls (11,6.5) .. (12.5,8);
\draw (12.5,6) .. controls (12,6.5) .. (12.5,7);
\node[] at (6.4,-0.6) {$\bar{h}_2\bar{h}_1\bar{h}_3$};
\draw (5,0) .. controls (6,1.5) .. (5,3);
\draw (5,1) .. controls (5.5,1.5) .. (5,2);
\draw (7.5,0) .. controls (7,0.5) .. (7.5,1);
\draw (7.5,2) .. controls (7,2.5) .. (7.5,3);
\node[] at (11.4,-0.6) {$\bar{h}_2\bar{h}_1\bar{h}_3\bar{h}_2$};
\draw (10,0) .. controls (11,1.5) .. (10,3);
\draw (10,1) .. controls (10.5,1.5) .. (10,2);
\draw (12.5,0) .. controls (11.5,1.5) .. (12.5,3);
\draw (12.5,1) .. controls (12,1.5) .. (12.5,2);
\end{tikzpicture}
\end{center}
\caption{The nonidentity diagrams in $\mathcal{J}_4$}\label{fig:jones}
\end{figure}
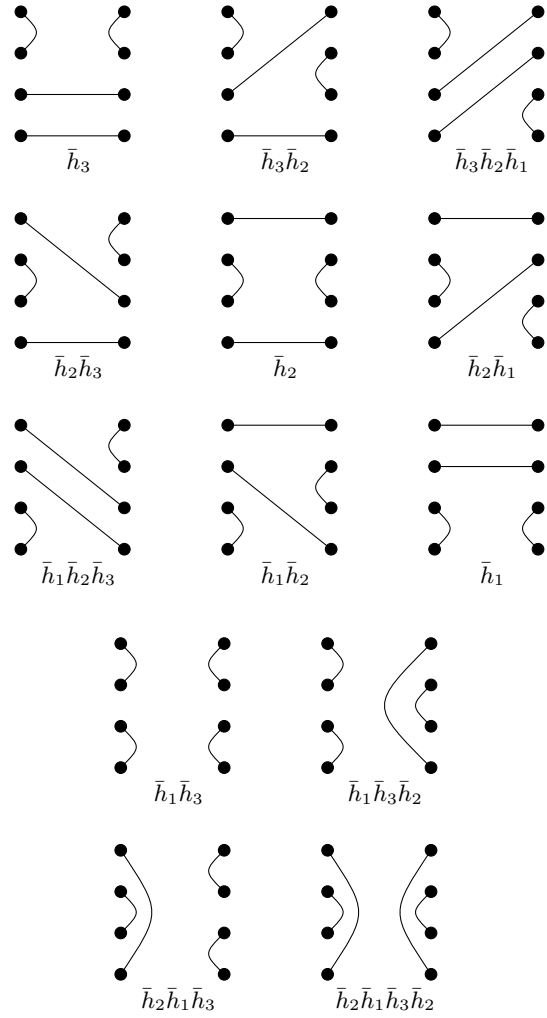

As a warm-up for our core results, we prove here a structure property of the monoid $\mathcal{J}_4$. This property quickly leads to a polynomial time algorithm for \textsc{Check-Id}($\mathcal{J}_4$).

Let $\mathcal{J}_4^\flat$ be the ideal of $\mathcal{J}_4$ consisting of its nonidentity diagrams, that is, of the 13 diagrams shown in Fig.~\ref{fig:jones}. We consider the following `cutting' map $\mathfrak{c}\colon\mathcal{J}_4^\flat\to\mathcal{J}_4^\flat$: if a diagram has no $t$-wires, $\mathfrak{c}$ fixes it; if a diagram has two $t$-wires, $\mathfrak{c}$ cuts the $t$-wires and then connects the loose ends, forming one new $\ell$-wire and one new $r$-wire, see Fig.~\ref{fig:cutting} for an illustration. More formally, the action of $\mathfrak{c}$ on a diagram with two $t$-wires amounts to:
 \begin{itemize}
   \item connecting the left points of the $t$-wires with an $\ell$-wire;
   \item connecting the right points of the $t$-wires with an $r$-wire;
   \item removing the $t$-wires.
\end{itemize}
Observe that the above operations make sense for diagrams with two $t$-wires in the Jones monoid $\mathcal{J}_n$ for every even $n\ge4$.

For the nine diagrams with two $t$-wires in the $3\times 3$-matrix in the upper half of Fig.~\ref{fig:jones}, the effect of the map $\mathfrak{c}$ can be described as follows:
\begin{itemize}
\item each of the four corner diagrams is sent to $\bar{h}_1\bar{h}_3$;
\item each of the two extreme diagrams in the middle row (column) is sent to $\bar{h}_2\bar{h}_1\bar{h}_3$ (respectively, $\bar{h}_1\bar{h}_3\bar{h}_2$);
\item the central diagram is sent to $\bar{h}_2\bar{h}_1\bar{h}_3\bar{h}_2$.
\end{itemize}

\begin{figure}[htb]
\begin{center}
\begin{tikzpicture}
[scale=0.55]
\foreach \x in {1,3.5} \foreach \y in {0,1,2,3} \filldraw (\x,\y) circle (4pt);
\draw (1,3) -- (3.5,1);
\draw (1,0) -- (3.5,0);
\draw (1,2) .. controls (1.5,1.5) .. (1,1);
\draw (3.5,2) .. controls (3,2.5) .. (3.5,3);
\draw[red,thick,->] (4,1.5) to (5.5,1.5);
\foreach \x in {6,8.5} \foreach \y in {0,1,2,3} \filldraw (\x,\y) circle (4pt);
\draw (6,3) -- (8.5,1);
\draw (6,0) -- (8.5,0);
\draw (6,2) .. controls (6.5,1.5) .. (6,1);
\draw (8.5,2) .. controls (8,2.5) .. (8.5,3);
\draw[red,thick,dashed] (7.25,-0.5) -- (7.25,3.5);
\draw[red,thick,->] (9,1.5) to (10.5,1.5);
\foreach \x in {11,13.5} \foreach \y in {0,1,2,3} \filldraw (\x,\y) circle (4pt);
\draw (11,3) -- (12,2.1);
\draw (13.5,1) -- (12.5,1.8);
\draw (11,0) -- (12,0);
\draw (13.5,0) -- (12.5,0);
\draw (11,2) .. controls (11.5,1.5) .. (11,1);
\draw (13.5,2) .. controls (13,2.5) .. (13.5,3);
\draw[red,thick,->] (14,1.5) to (15.5,1.5);
\foreach \x in {16,18.5} \foreach \y in {0,1,2,3} \filldraw (\x,\y) circle (4pt);
\draw (16,2) .. controls (16.5,1.5) .. (16,1);
\draw (16,3) .. controls (17,1.5) .. (16,0);
\draw (18.5,0) .. controls (18,0.5) .. (18.5,1);
\draw (18.5,2) .. controls (18,2.5) .. (18.5,3);
\end{tikzpicture}
\end{center}
\caption{The cutting map $\mathfrak{c}$ on $\mathcal{J}_4^\flat$}\label{fig:cutting}
\end{figure}
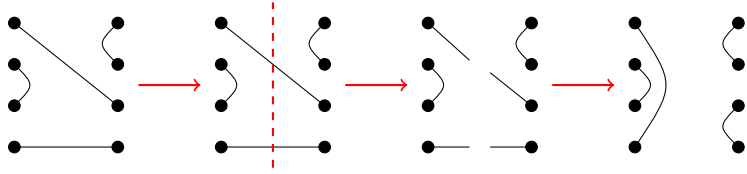

\begin{lemma}
\label{lem:cutting}
The map $\mathfrak{c}\colon\mathcal{J}_4^\flat\to\mathcal{J}_4^\flat$ is an endomorphism of $\mathcal{J}_4^\flat$.
\end{lemma}

\begin{proof}
The lemma can be verified by a direct computation. We prefer a more geometric argument since it also works in a more general situation.

Let $\xi\in\mathcal{J}_4^\flat$ have two $t$-wires. Then the $\ell$-wires of $\xi\mathfrak{c}$ are:
\begin{equation}
\label{eq:wires}
\begin{aligned}
&\text{the $\ell$-wire of $\xi$, and}\\
&\text{the $\ell$-wire that connects the left points of the $t$-wires of $\xi$.}
\end{aligned}
\end{equation}
Now consider an arbitrary diagram $\eta\in\mathcal{J}_4^\flat$. The product $\xi\mathfrak{c}\cdot\eta\mathfrak{c}$ has the same $\ell$-wires~\eqref{eq:wires}. The product $\xi\eta$ has either two or no $t$-wires. In the latter case its $\ell$-wires coincide with those in~\eqref{eq:wires}. If $\xi\eta$ has two $t$-wires, their left points are the same as the left points of the $t$-wires of $\xi$ whence the $\ell$-wires of $(\xi\eta)\mathfrak{c}$ are those in~\eqref{eq:wires} again.

We see that the $\ell$-wires of $\xi\mathfrak{c}\cdot\eta\mathfrak{c}$ and $(\xi\eta)\mathfrak{c}$ are equal. By symmetry, $\xi\mathfrak{c}\cdot\eta\mathfrak{c}$ and $(\xi\eta)\mathfrak{c}$ have the same $r$-wires as well. Hence, $\xi\mathfrak{c}\cdot\eta\mathfrak{c}=(\xi\eta)\mathfrak{c}$.\qed
\end{proof}

\begin{remark}
Let $n\ge 4$ be an even number. The set $\mathcal{J}_n^\flat$ of all diagrams with at most two $t$-wires forms a subsemigroup in the Jones  monoid $\mathcal{J}_n$. The proof of Lemma~\ref{lem:cutting} shows that the cutting map is an endomorphism of $\mathcal{J}_n^\flat$.
\end{remark}

An endomorphism that fixes each element in its image is called a \emph{retraction}. We need the following folklore result of semigroup theory.

\begin{lemma}
\label{lem:retract}
If $\varphi$ is a retraction of a semigroup $\mathcal{S}$ such that $\mathcal{S}\varphi$ is an ideal of $\mathcal{S}$, then $\mathcal{S}$ is isomorphic to a subdirect product of the ideal $\mathcal{S}\varphi$ with the Rees quotient $\mathcal{S}/\mathcal{S}\varphi$.\qed
\end{lemma}

\begin{proposition}
\label{prop:structure}
The semigroup $\mathcal{J}_4^\flat$ is isomorphic to a subdirect product of a $2\times 2$ rectangular band with the \Rms\ $\mathcal{M}_3:=\mathcal{M}^0\left(\{1,2,3\},\mathcal{E},\{1,2,3\};\Bigl(\begin{smallmatrix}e&e&0\\e&e&e\\0&e&e\end{smallmatrix}\Bigr)\right)$ over the one-element group $\mathcal{E}=\{e\}$.
\end{proposition}

\begin{proof}
By the definition of the map $\mathfrak{c}\colon\mathcal{J}_4^\flat\to\mathcal{J}_4^\flat$, its image is the set $\mathcal{I}_4$ consisting of the four diagrams in $\mathcal{J}_4^\flat$ that have no $t$-wires. Since $\mathfrak{c}$ fixes each diagram in $\mathcal{I}_4$ and is an endomorphism by Lemma~\ref{lem:cutting}, $\mathfrak{c}$ is a retraction. Clearly, $\mathcal{I}_4$ is an ideal of $\mathcal{J}_4^\flat$. We are in a position to apply
Lemma~\ref{lem:retract}, which yields that $\mathcal{J}_4^\flat$ is isomorphic to a subdirect product of the ideal $\mathcal{I}_4$ with the Rees quotient $\mathcal{J}_4^\flat/\mathcal{I}_4$.

Obviously, $\mathcal{I}_4$ is a $2\times 2$ rectangular band. As for the Rees quotient $\mathcal{J}_4^\flat/\mathcal{I}_4$, it can be mapped onto the \Rms\ $\mathcal{M}_3$ as follows: the zero of $\mathcal{J}_4^\flat/\mathcal{I}_4$ is sent to 0 and the diagram in the $i$-th row and $j$-th column of the $3\times 3$-matrix in the upper half of Fig.~\ref{fig:jones} is sent to the triple $(i,e,j)$. One can directly verify that the bijection defined this way is an isomorphism between $\mathcal{J}_4^\flat/\mathcal{I}_4$ and $\mathcal{M}_3$.\qed
\end{proof}

Clearly, an identity holds in a subdirect product if and only if it holds in every factor of the product. Thus, Proposition~\ref{prop:structure} implies that an identity holds in the semigroup $\mathcal{J}_4^\flat$ if and only if it holds in both $\mathcal{I}_4$ and $\mathcal{M}_3$. Observe that the triples $(i,e,j)\in\mathcal{M}_3$ with $i,j\in\{1,2\}$ form a $2\times 2$ rectangular band. We see that $\mathcal{I}_4$ is isomorphic to a
subsemigroup in $\mathcal{M}_3$, and thus, satisfies all identities of the latter semigroup. Hence,  the semigroups $\mathcal{J}_4^\flat$ and $\mathcal{M}_3$ are \emph{equationally equivalent}, that is, they satisfy the same identities.

A combinatorial characterization of the identities of $\mathcal{M}_3$ is known. Namely, it easily follows from a result by \cite{Tr81} that an identity $w\bumpeq w'$ holds in $\mathcal{M}_3$ if and only if the words $w$ and $w'$ satisfy the conditions (a) and (b) of Proposition~\ref{prop:rms} along with the following condition:
\begin{itemize}
\item[(c')] each word of length $2$ occurs in $w$ if and only if it occurs in $w'$.
\end{itemize}

It is easy to characterize identities of a semigroup $\mathcal{S}$ that are inherited by the monoid $\mathcal{S}^1$. Namely, for a word $w\in X^+$ and a proper subset $Y$ of $\al(w)$, denote by $w_Y$ the word obtained from $w$ by removing all occurrences of the letters in $Y$. The following observation is another part of semigroup folklore.

\begin{lemma}
\label{lem:monoid}
Let $\mathcal{S}$ be a semigroup. The monoid $\mathcal{S}^1$ satisfies an identity $w\bumpeq w'$ with $\al(w)=\al(w')$ if and only if the identity $w_Y\bumpeq w'_Y$ holds in $\mathcal{S}$ for each $Y\subset\al(w)$.\qed
\end{lemma}
The restriction $\al(w)=\al(w')$ in Lemma~\ref{lem:monoid} is not essential for what follows because a monoid satisfying a semigroup identity $w\bumpeq w'$ with $\al(w)\ne\al(w')$ is easily seen to be a group while monoids we consider are very far from being groups.

Lemma~\ref{lem:monoid} readily implies that if two semigroups $\mathcal{S}_1$ and $\mathcal{S}_2$ are equationally equivalent, so are the monoids $\mathcal{S}_1^1$ and $\mathcal{S}_2^1$. Hence, the Jones monoid $\mathcal{J}_4$ is equationally equivalent to the monoid $\mathcal{M}_3^1$. Summing up, we get the following characterization of the identities of the monoid $\mathcal{J}_4$.
\begin{theorem}
\label{thm:jones description}
An identity $w\bumpeq w'$ holds in the Jones monoid $\mathcal{J}_4$ if and only if $\al(w)=\al(w')$ and, for each $Y\subset\al(w)$, the words $u:=w_Y$ and $u':=w'_Y$ satisfy the following three conditions:
\begin{itemize}
\item[\emph{(a)}] the first letter of $u$ is the same as the first letter of $u'$;
\item[\emph{(b)}] the last letter of $u$ is the same as the last letter of $u'$;
\item[\emph{(c')}] each word of length $2$ occurs in $u$ if and only if it occurs in $u'$.\qed
\end{itemize}
\end{theorem}

\begin{remark}
It is not immediately clear whether Theorem~\ref{thm:jones description} provides a polynomial time algorithm for \textsc{Check-Id}($\mathcal{J}_4$) since a brute force verification of the conditions (a)--(c') for every proper subset of the set $\al(w)$ requires exponential in $|\al(w)|$ time. In fact, there exist examples of finite semigroups $\mathcal{S}$ such that \textsc{Check-Id}($\mathcal{S}$) is in $\mathsf{P}$ while \textsc{Check-Id}($\mathcal{S}^1$) is $\mathsf{coNP}$-complete, see, e.g., \citep{Se05,Kl09}. However, \citet{SS06} have proved that one can verify the conditions (a)--(c') in polynomial in $|ww'|$ time. Thus, \textsc{Check-Id}($\mathcal{J}_4$) lies in $\mathsf{P}$. Moreover, using methods developed in \citep{Chen20}, one can check whether or nor the monoid $\mathcal{J}_4$ satisfies an identity $w\bumpeq w'$ with $|\al(w)|=k$ and $|ww'|=n$ in $O(kn\log(kn))$ time.
\end{remark}

\section{Structure of $\widehat{\mathcal{K}}_4$ and identities of $\mathcal{K}_4$}
\label{sec:wire}

We are ready to attack the identity checking problem for the Kauffman monoid $\mathcal{K}_4$. We approach the problem via a structure property as we did in Section~\ref{sec:j4} for \textsc{Check-Id}($\mathcal{J}_4$). We start with lifting the cutting map $\mathfrak{c}$ from Jones to Kauffman monoids; technically, it is more convenient to lift the map to the extended Kauffman monoid $\widehat{\mathcal{K}}_4$.

Let $\widehat{\mathcal{K}}_4^\flat$ be the ideal of $\widehat{\mathcal{K}}_4$ consisting of all diagrams with at most two $t$-wires; in other words, $\widehat{\mathcal{K}}_4^\flat$ is nothing but the preimage of $\mathcal{J}_4^\flat$ under the erasing map $\xi\mapsto\bar{\xi}$. We define a map $\mathfrak{C}\colon\widehat{\mathcal{K}}_4^\flat\to\widehat{\mathcal{K}}_4^\flat$ as follows: $\mathfrak{C}$ fixes each diagram that has no $t$-wires; if a diagram has two $t$-wires, $\mathfrak{C}$ cuts out the middle of each $t$-wire and then connects the loose ends, forming one new $\ell$-wire, one new $r$-wire, and a new \textbf{negative} circle, which then annihilates with a positive circle provided the initial diagram had positive circles. See Fig.~\ref{fig:cutting} for an illustration.
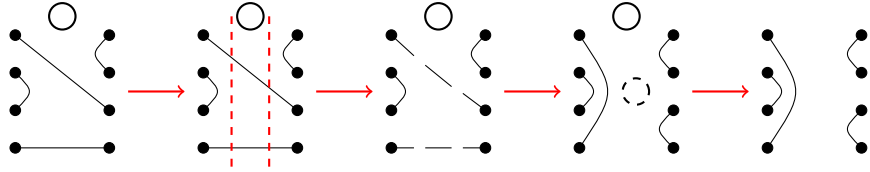
\begin{figure}[htb]
\begin{center}
\begin{tikzpicture}
[scale=0.5]
\foreach \x in {1,3.5} \foreach \y in {0,1,2,3} \filldraw (\x,\y) circle (4pt);
\draw (1,3) -- (3.5,1);
\draw (1,0) -- (3.5,0);
\draw (1,2) .. controls (1.5,1.5) .. (1,1);
\draw (3.5,2) .. controls (3,2.5) .. (3.5,3);
\foreach \x in {2.25,7.25,12.25,17.25} \draw[thick] (\x,3.5) circle (10pt);
\draw[red,thick,->] (4,1.5) to (5.5,1.5);
\foreach \x in {6,8.5} \foreach \y in {0,1,2,3} \filldraw (\x,\y) circle (4pt);
\draw (6,3) -- (8.5,1);
\draw (6,0) -- (8.5,0);
\draw (6,2) .. controls (6.5,1.5) .. (6,1);
\draw (8.5,2) .. controls (8,2.5) .. (8.5,3);
\draw[red,thick,dashed] (6.75,-0.5) -- (6.75,3.5);
\draw[red,thick,dashed] (7.75,-0.5) -- (7.75,3.5);
\draw[red,thick,->] (9,1.5) to (10.5,1.5);
\foreach \x in {11,13.5} \foreach \y in {0,1,2,3} \filldraw (\x,\y) circle (4pt);
\draw (11,3) -- (11.6,2.45);
\draw (13.5,1) -- (12.9,1.45);
\draw (11.9,2.2) -- (12.6,1.63);
\draw (11,0) -- (11.6,0);
\draw (13.5,0) -- (12.9,0);
\draw (11.9,0) -- (12.6,0);
\draw (11,2) .. controls (11.5,1.5) .. (11,1);
\draw (13.5,2) .. controls (13,2.5) .. (13.5,3);
\draw[red,thick,->] (14,1.5) to (15.5,1.5);
\foreach \x in {16,18.5} \foreach \y in {0,1,2,3} \filldraw (\x,\y) circle (4pt);
\draw (16,2) .. controls (16.5,1.5) .. (16,1);
\draw (16,3) .. controls (17,1.5) .. (16,0);
\draw (18.5,0) .. controls (18,0.5) .. (18.5,1);
\draw (18.5,2) .. controls (18,2.5) .. (18.5,3);
\draw[thick,dashed] (17.5,1.5) circle (10pt);
\draw[red,thick,->] (19,1.5) to (20.5,1.5);
\foreach \x in {21,23.5} \foreach \y in {0,1,2,3} \filldraw (\x,\y) circle (4pt);
\draw (21,2) .. controls (21.5,1.5) .. (21,1);
\draw (21,3) .. controls (22,1.5) .. (21,0);
\draw (23.5,0) .. controls (23,0.5) .. (23.5,1);
\draw (23.5,2) .. controls (23,2.5) .. (23.5,3);
\end{tikzpicture}
\end{center}
\caption{The cutting map $\mathfrak{C}$ on $\widehat{\mathcal{K}}_4^\flat$; solid/dashed circles are positive/negative}\label{fig:cutting_new}
\end{figure}

Formally, if a diagram $\xi\in\widehat{\mathcal{K}}_4^\flat$ corresponds to the pair $\bigl(\bar{\xi},\|\xi\|\bigr)\in\mathcal{J}_4^\flat\times\mathbb{Z}$, then $\xi\mathfrak{C}$ is the diagram corresponding to the pair $\bigl(\bar{\xi}\mathfrak{c},\|\xi\|-1\bigr)$ if $\xi$ has two $t$-wires and $\xi\mathfrak{C}=\xi$ otherwise. Observe that $\overline{\xi\mathfrak{C}}=\bar{\xi}\mathfrak{c}$ for every $\xi\in\widehat{\mathcal{K}}_4^\flat$.

\begin{lemma}
\label{lem:cutting_new}
The map $\mathfrak{C}\colon\widehat{\mathcal{K}}_4^\flat\to\widehat{\mathcal{K}}_4^\flat$ is an endomorphism of $\widehat{\mathcal{K}}_4^\flat$.
\end{lemma}

\begin{proof}
We have to show that $\xi\mathfrak{C}\cdot\eta\mathfrak{C}=(\xi\eta)\mathfrak{C}$ for arbitrary diagrams $\xi,\eta\in\widehat{\mathcal{K}}_4^\flat$. If both $\xi$ and $\eta$ have no $t$-wires, so does $\xi\eta$, and the required equality clearly holds. Thus, we may assume that at least one of the diagrams has two $t$-wires. Due to the symmetry, it is sufficient to analyze the situation when $\xi$ has two $t$-wires.

In terms of the coordinatization of $\widehat{\mathcal{K}}_4$, the diagram $\xi\mathfrak{C}\cdot\eta\mathfrak{C}$ corresponds to the pair $\bigl(\overline{\xi\mathfrak{C}\cdot\eta\mathfrak{C}},\|\xi\mathfrak{C}\cdot\eta\mathfrak{C}\|\bigr)$ while the pair corresponding to $(\xi\eta)\mathfrak{C}$ is $\bigl(\overline{(\xi\eta)\mathfrak{C}},\|(\xi\eta)\mathfrak{C}\|\bigr)$. The equality of the first entries of these pairs easily follows from Lemma~\ref{lem:cutting}. Indeed,
\begin{align*}
  \overline{\xi\mathfrak{C}\cdot\eta\mathfrak{C}}&=\overline{\xi\mathfrak{C}}\cdot\overline{\eta\mathfrak{C}}=\bar{\xi}\mathfrak{c}\cdot\bar{\eta}\mathfrak{c}&&\text{since $\xi\mapsto\bar{\xi}$ is a homomorphism}\\
  &=(\bar{\xi}\bar{\eta})\mathfrak{c}&&\text{by Lemma~\ref{lem:cutting}}\\
  &=(\overline{\xi\eta})\mathfrak{c}=\overline{(\xi\eta)\mathfrak{C}}&&\text{since $\xi\mapsto\bar{\xi}$ is a homomorphism}.
\end{align*}
Thus, it remains to compute the numbers of circles in $\xi\mathfrak{C}\cdot\eta\mathfrak{C}$ and in $(\xi\eta)\mathfrak{C}$ and to verify that these numbers are equal, that is,
$\|\xi\mathfrak{C}\cdot\eta\mathfrak{C}\|=\|(\xi\eta)\mathfrak{C}\|$. The following aims to present the computation in a compact way.

From~\eqref{eq:coordinates}, we see that $\|\xi\mathfrak{C}\cdot\eta\mathfrak{C}\|=\|\xi\mathfrak{C}\|+\|\eta\mathfrak{C}\|+\langle\overline{\xi\mathfrak{C}},\overline{\eta\mathfrak{C}}\rangle$. Since $\overline{\xi\mathfrak{C}}=\bar{\xi}\mathfrak{c}$ and $\overline{\eta\mathfrak{C}}=\bar{\eta}\mathfrak{c}$, the desired equality can be rewritten as
\begin{equation}
\label{eq:desired}
\|\xi\mathfrak{C}\|+\|\eta\mathfrak{C}\|+\langle\bar{\xi}\mathfrak{c},\bar{\eta}\mathfrak{c}\rangle=\|(\xi\eta)\mathfrak{C}\|.
\end{equation}

We say that a diagram $\gamma\in\mathcal{J}_4^\flat$ \emph{matches} a diagram $\delta\in\mathcal{J}_4^\flat$ if for every $r$-wire $\{i',j'\}$ of $\gamma$, the set $\{i,j\}$ occurs as an $\ell$-wire in $\delta$. (Observe that we do not require the opposite: $\delta$ may have an $\ell$-wire $\{s,t\}$, say, such that $\{s',t'\}$ is not an $r$-wire in $\gamma$.) Clearly, gluing $\{i',j'\}$ with $\{i,j\}$ creates a circle when the product $\gamma\delta$ is being formed.

We split the verification of \eqref{eq:desired} into three cases. Each of these cases covers a certain number of pairs $(\bar{\xi},\bar{\eta})$  amongst $9\times 13=117$ pairs that are subject to checking. (The assumption that $\xi$ has two $t$-wires restricts the choice of $\bar{\xi}$ to the nine diagrams in the upper half of Fig.~\ref{fig:jones} while $\bar{\eta}$ can be any of the 13 diagrams from $\mathcal{J}_4^\flat$.) The reader may find it helpful to trace how the argument of each case works on a typical example; for this, we indicate such examples, after stating the conditions of the cases.

\medskip

\noindent\textbf{\emph{Case 1}}: $\bar{\xi}$ matches $\bar{\eta}$.

Here typical representatives are the pairs $(\bar{h}_1,\bar{h}_1\bar{h}_2)$ and $(\bar{h}_1,\bar{h}_1\bar{h}_3)$.

\smallskip

The condition that $\bar{\xi}$ matches $\bar{\eta}$ means that $\langle\bar{\xi},\bar{\eta}\rangle=1$. Further, it is easy to see that $\bar{\xi}\mathfrak{c}$ matches $\bar{\eta}\mathfrak{c}$ whence $\langle\bar{\xi}\mathfrak{c},\bar{\eta}\mathfrak{c}\rangle=2$. We have $\|\xi\mathfrak{C}\|=\|\xi\|-1$ as $\xi$ has two $t$-wires. If $\eta$ also has two $t$-wires, then $\|\eta\mathfrak{C}\|=\|\eta|-1$. Thus, computing the left hand side of \eqref{eq:desired} yields
\[
\|\xi\mathfrak{C}\|+\|\eta\mathfrak{C}\|+\langle\bar{\xi}\mathfrak{c},\bar{\eta}\mathfrak{c}\rangle=(\|\xi\|-1)+(\|\eta\|-1)+2=\|\xi\|+\|\eta\|.
\]
Besides that, the condition that $\bar{\xi}$ matches $\bar{\eta}$ implies that the $t$-wires of $\xi$ and $\eta$ combine and provide two $t$-wires in $\xi\eta$. Using this and~\eqref{eq:coordinates}, we get
\[
\|(\xi\eta)\mathfrak{C}\|=\|\xi\eta\|-1=\|\xi\|+\|\eta\|+\langle\bar{\xi},\bar{\eta}\rangle-1=\|\xi\|+\|\eta\|+1-1=\|\xi\|+\|\eta\|.
\]
We conclude that the equality \eqref{eq:desired} holds.

Now assume that $\eta$ has no $t$-wires. Then $\|\eta\mathfrak{C}\|=\|\eta|$, whence the left hand side of \eqref{eq:desired} is equal to $\|\xi\|+\|\eta\|+1$. However, in this subcase, the product $\xi\eta$ also omits $t$-wires and $\|(\xi\eta)\mathfrak{C}\|=\|\xi\eta\|=\|\xi\|+\|\eta\|+1$, too. Thus, the equality \eqref{eq:desired} persists.

\medskip

\noindent\textbf{\emph{Case 2}}: $\bar{\xi}$ does not match $\bar{\eta}$ but $\bar{\xi}\mathfrak{c}$ matches $\bar{\eta}\mathfrak{c}$.

Here a typical representative is the pair $(\bar{h}_1,\bar{h}_3)$.

\smallskip

Case 2 is only possible if $\eta$ has two $t$-wires whence $\|\eta\mathfrak{C}\|=\|\eta|-1$. We have $\langle\bar{\xi},\bar{\eta}\rangle=0$ but $\langle\bar{\xi}\mathfrak{c},\bar{\eta}\mathfrak{c}\rangle=2$. Thus, the left hand side of \eqref{eq:desired} is
\[
\|\xi\mathfrak{C}\|+\|\eta\mathfrak{C}\|+\langle\bar{\xi}\mathfrak{c},\bar{\eta}\mathfrak{c}\rangle=(\|\xi\|-1)+(\|\eta\|-1)+2=\|\xi\|+\|\eta\|.
\]
Further, under the conditions of Case~2, $\xi\eta$ cannot possess $t$-wires. Therefore, $\|(\xi\eta)\mathfrak{C}\|=\|\xi\eta\|=\|\xi\|+\|\eta\|$, and the equality \eqref{eq:desired} holds.

\medskip

\noindent\textbf{\emph{Case 3}}: $\bar{\xi}\mathfrak{c}$ does not match $\bar{\eta}\mathfrak{c}$.

Here typical representatives are the pairs $(\bar{h}_1,\bar{h}_2)$ and $(\bar{h}_1,\bar{h}_2\bar{h}_1\bar{h}_3)$.

\smallskip

Since $\bar{\xi}\mathfrak{c}$ does not match $\bar{\eta}\mathfrak{c}$, we have $\langle\bar{\xi}\mathfrak{c},\bar{\eta}\mathfrak{c}\rangle=1$. In addition, $\bar{\xi}$ cannot match $\bar{\eta}$ whence $\langle\bar{\xi},\bar{\eta}\rangle=0$. If $\eta$ has two $t$-wires, $\|\eta\mathfrak{C}\|=\|\eta|-1$ and the left hand side of \eqref{eq:desired} becomes
\[
\|\xi\mathfrak{C}\|+\|\eta\mathfrak{C}\|+\langle\bar{\xi}\mathfrak{c},\bar{\eta}\mathfrak{c}\rangle=(\|\xi\|-1)+(\|\eta\|-1)+1=\|\xi\|+\|\eta\|-1.
\]
If the $r$-wire of $\xi$ is $\{i',j'\}$, the $\ell$-wire of $\eta$ must be $\{j,k\}$ for some $k\ne i$. The set $\{1,2,3,4\}\setminus\{i,j,k\}$ consists of a unique number $h$, say. Then one of the $t$-wires of $\xi$ has $h'$ as its right point while one of the $t$-wires of $\eta$ has $h$ as its left point, and we see that $\xi\eta$ has got a $t$-wire. From this and~\eqref{eq:coordinates}, we compute
\[
\|(\xi\eta)\mathfrak{C}\|=\|\xi\eta\|-1=\|\xi\|+\|\eta\|+\langle\bar{\xi},\bar{\eta}\rangle-1=\|\xi\|+\|\eta\|-1,
\]
whence the equality \eqref{eq:desired} holds.

Finally, consider the subcase when $\eta$ has no $t$-wires. Then $\|\eta\mathfrak{C}\|=\|\eta\|$ and the left hand side of \eqref{eq:desired} becomes $\|\xi\mathfrak{C}\|+\|\eta\mathfrak{C}\|+\langle\bar{\xi}\mathfrak{c},\bar{\eta}\mathfrak{c}\rangle=(\|\xi\|-1)+\|\eta\|+1=\|\xi\|+\|\eta\|$. Of course, if $\eta$ omits $t$-wires, so does $\xi\eta$, whence  $\|(\xi\eta)\mathfrak{C}\|=\|\xi\eta\|=\|\xi\|+\|\eta\|$, and the equality \eqref{eq:desired} holds again.\qed
\end{proof}

\begin{remark}
When we introduced the extended Kauffman monoids $\widehat{\mathcal{K}}_n$, we said that they are easier to deal with, compared with the `standard' Kauffman monoids $\mathcal{K}_n$. Lemma~\ref{lem:cutting_new} provides a supporting evidence for this claim. Indeed, it is not clear if the semigroup $\mathcal{K}_4^\flat$ consisting of diagrams with at most two $t$-wires from $\mathcal{K}_4$ admits any `nice' endomorphism similar to the cutting map $\mathfrak{C}\colon\widehat{\mathcal{K}}_4^\flat\to\widehat{\mathcal{K}}_4^\flat$. We mention in passing that working with the monoid $\widehat{\mathcal{K}}_3$ rather than $\mathcal{K}_3$ would have somewhat simplified also the proofs of the main results in~\citep{Chen20}.
\end{remark}

We proceed with an analogue of Proposition~\ref{prop:structure}. Recall that $\mathbb{C}_\infty$ stands for the infinite cyclic group. As above, we fix a generator $c$ of $\mathbb{C}_\infty$ and denote by $e$ the identity element of the group. Now consider two \Rms{s} over $\mathbb{C}_\infty$:
\begin{itemize}
\item $\mathcal{RC}_2:=\mathcal{M}\left(\{1,2\},\mathbb{C}_\infty,\{1,2\};\Bigl(\begin{smallmatrix}c^2&c\\c&c^2\end{smallmatrix}\Bigr)\right)$,
\item $\mathcal{MC}_3:=\mathcal{M}^0\left(\{1,2,3\},\mathbb{C}_\infty,\{1,2,3\};\Bigl(\begin{smallmatrix}c&e&0\\e&c&e\\0&e&c\end{smallmatrix}\Bigr)\right)$.
\end{itemize}

\begin{proposition}
\label{prop:structure_new}
The semigroup $\widehat{\mathcal{K}}_4^\flat$ is isomorphic to a subdirect product of the \Rms{s} $\mathcal{RC}_2$ and $\mathcal{MC}_3$.
\end{proposition}

\begin{proof}
By the definition of the map $\mathfrak{C}\colon\widehat{\mathcal{K}}_4^\flat\to\widehat{\mathcal{K}}_4^\flat$, its image is the set $\widehat{\mathcal{I}}_4$ consisting of the diagrams in $\widehat{\mathcal{K}}_4^\flat$ that have no $t$-wires. Since $\mathfrak{C}$ fixes each diagram in $\widehat{\mathcal{I}}_4$ and is an endomorphism by Lemma~\ref{lem:cutting_new}, $\mathfrak{C}$ is a retraction. Since $\widehat{\mathcal{I}}_4$ is an ideal of $\widehat{\mathcal{K}}_4^\flat$, Lemma~\ref{lem:retract} applies, providing a decomposition of $\widehat{\mathcal{K}}_4^\flat$ into a subdirect product of $\widehat{\mathcal{I}}_4$ with the Rees quotient $\widehat{\mathcal{K}}_4^\flat/\widehat{\mathcal{I}}_4$.

It remains to show that $\widehat{\mathcal{I}}_4$ is isomorphic to $\mathcal{RC}_2$ and $\widehat{\mathcal{K}}_4^\flat/\widehat{\mathcal{I}}_4$ is isomorphic to $\mathcal{MC}_3$. Both isomorphisms are easy to describe in terms
of the coordinatization of diagrams from $\widehat{\mathcal{K}}_4$ by pairs from $\mathcal{J}_4\times\mathbb{Z}$. If $\eta\in\widehat{\mathcal{I}}_4$ corresponds to the pair $(\bar{\eta},m)\in\mathcal{J}_4\times\mathbb{Z}$ and the diagram $\bar{\eta}$ occurs in the $i$-th row and $j$-th column of the $2\times 2$-matrix in the lower half of Fig.~\ref{fig:jones}, then $\eta$ is sent to the triple $(i,c^m,j)\in\mathcal{RC}_2$. Similarly, if $\xi\in\widehat{\mathcal{K}}_4^\flat\setminus\widehat{\mathcal{I}}_4$ corresponds to the pair $(\bar{\xi},n)\in\mathcal{J}_4\times\mathbb{Z}$ and the diagram $\bar{\xi}$ occurs in the $k$-th row and $\ell$-th column of the of the $3\times 3$-matrix in the upper half of Fig.~\ref{fig:jones}, then $\xi$ is sent to the triple $(k,c^n,\ell)\in\mathcal{MC}_3$. Finally, the zero of the Rees quotient $\widehat{\mathcal{K}}_4^\flat/\widehat{\mathcal{I}}_4$ is sent to $0\in\mathcal{MC}_3$. Thus, we have got a bijection between $\widehat{\mathcal{I}}_4$ and $\mathcal{RC}_2$, as well as a bijection between $\widehat{\mathcal{K}}_4^\flat/\widehat{\mathcal{I}}_4$ and $\mathcal{MC}_3$. The verification that these bijections constitute semigroup isomorphisms is immediate.\qed
\end{proof}

Recall the description of the identities of $\mathcal{K}_3$ from~\citep{Chen20}.

\begin{theorem}
\label{thm:description}
An identity $w\bumpeq w'$ holds in the Kauffman monoid $\mathcal{K}_3$ if and only if $\al(w)=\al(w')$ and, for each $Y\subset\al(w)$, the words $u:=w_Y$ and $u':=w'_Y$ satisfy the following three conditions:
\begin{itemize}
\item[\emph{(a)}] the first letter of $u$ is the same as the first letter of $u'$;
\item[\emph{(b)}] the last letter of $u$ is the same as the last letter of $u'$;
\item[\emph{(c)}] for each word of length $2$, the number of its occurrences in $u$ is the same as the number of its occurrences in $u'$.\qed
\end{itemize}
\end{theorem}

We are ready to state and to prove our main result.

\begin{theorem}
\label{thm:equivalence}
The Kauffman monoids $\mathcal{K}_3$ and $\mathcal{K}_4$ are equationally equivalent.
\end{theorem}

\begin{proof}
The monoid $\mathcal{K}_3$ naturally embeds into $\mathcal{K}_4$: the submonoid of $\mathcal{K}_4$ generated by the hooks $h_1,h_2$ and the circle $c$ is isomorphic to $\mathcal{K}_3$. Therefore, every identity that holds in $\mathcal{K}_4$ must hold in $\mathcal{K}_3$. In order to show the converse, we employ Theorem~\ref{thm:description}. Namely, we are going to verify that every identity $w\bumpeq w'$ that satisfies the conditions of Theorem~\ref{thm:description} holds in the extended Kauffman monoid $\widehat{\mathcal{K}}_4$. Since $\mathcal{K}_4$ embeds into $\widehat{\mathcal{K}}_4$, this will prove the equational equivalence of $\mathcal{K}_3$ with $\mathcal{K}_4$, and moreover, with $\widehat{\mathcal{K}}_4$.

We have to check that $w\varphi=w'\varphi$ for an arbitrary homomorphism $\varphi\colon X^+\to\widehat{\mathcal{K}}_4$. Clearly, $\widehat{\mathcal{K}}_4$ is the disjoint union of its group of units $H$ generated (as a semigroup) by $c$ and $d$ and the ideal $\widehat{\mathcal{K}}_4^\flat$. Let $Y:=\{y\in\al(w)\mid y\varphi\in H\}$. Since $cd=dc=1$, we write $c^{-1}$ for $d$, and for each $y\in Y$, we let $k_y\in\mathbb{Z}$ be such that $y\varphi=c^{k_y}$. Denote the sum $\sum_{y\in Y}\occ_y(w)k_y$ by $N_Y$. By Lemma~\ref{lem:balanced} we have $\occ_y(w)=\occ_y(w')$, whence the sum $\sum_{y\in Y}\occ_y(w')k_y$ is also equal to $N_Y$. If $Y=\al(w)$, we have $w\varphi=c^{N_Y}=w'\varphi$, and we are done.

Consider the situation where $Y\subset\al(w)$. Using the fact that the generators $c$ and $d$ commute with the hooks $h_1,h_2,h_3$, we can represent $w\varphi$ and $w'\varphi$ as $c^{N_Y}w_Y\varphi$ and $c^{N_Y}w'_Y\varphi$ respectively. Therefore it remains to verify that $w_Y\varphi=w'_Y\varphi$, and for this, it suffices to show that the identity $u\bumpeq u'$ with $u:=w_Y$ and $u':=w'_Y$ holds in the semigroup $\widehat{\mathcal{K}}_4^\flat$.
Since the words $u$ and $u'$ satisfy the conditions (a)--(c), the identity $u\bumpeq u'$ holds in every \Rms\ over an abelian group by Proposition~\ref{prop:rms}. In particular, $u\bumpeq u'$ holds in the semigroups  $\mathcal{RC}_2$ and $\mathcal{MC}_3$, and by Proposition~\ref{prop:structure_new} it holds also in $\widehat{\mathcal{K}}_4^\flat$, as required.\qed
\end{proof}

\begin{remark}
The result of Theorem~\ref{thm:equivalence} was unexpected for us since, informally speaking, the monoid $\mathcal{K}_4$ appeared to be much more complicated than its submonoid $\mathcal{K}_3$ and it was rather hard to believe that the $\mathcal{K}_4$ could inherit all identities of the submonoid. Observe that the Jones monoids $\mathcal{J}_3$ and $\mathcal{J}_4$ are not equationally equivalent: $\mathcal{J}_3$ satisfies the identity $x^2\bumpeq x$ that clearly fails in $\mathcal{J}_4$. Moreover, it follows from a result by \cite{Tr88} that the identities of  $\mathcal{J}_3$ and $\mathcal{J}_4$ are very different in a sense: there are uncountably many pairwise equationally non-equivalent semigroups whose identity sets strictly contain the identity set of $\mathcal{J}_4$ and are strictly contained in that of $\mathcal{J}_3$.  Theorem~\ref{thm:equivalence} makes a strong contrast to these facts.
\end{remark}

Using a suitable reformulation of Theorem~\ref{thm:description}, \citet[Section~2]{Chen20} have developed an algorithm that, given an identity $w\bumpeq w'$ with $|\al(w)|=k$ and $|ww'|=n$, verifies whether or nor the identity holds in the monoid $\mathcal{K}_3$ in $O(kn\log(kn))$ time. Theorem~\ref{thm:equivalence} implies that this algorithm can be used to check identities  in the monoid $\mathcal{K}_4$. In particular, we have the following fact.

\begin{corollary}
The problem \textsc{Check-Id}($\mathcal{K}_4$) lies in $\mathsf{P}$.\qed
\end{corollary}

It has been shown in \citep[Proposition 6]{Chen20} that the equational equivalence of $\mathcal{K}_3$ and $\mathcal{K}_4$ does not extend to the monoid $\mathcal{K}_5$. For instance, the identity $x^2yx\bumpeq xyx^2$, which holds in $\mathcal{K}_3$ and $\mathcal{K}_4$ by Theorems~\ref{thm:description} and~\ref{thm:equivalence}, fails in $\mathcal{K}_5$ under the substitution $x\mapsto h_1h_2h_3,\ y\mapsto h_4$. The proof in \citep{Chen20} relies on a normal form for the elements of the monoid $\mathcal{K}_n$ suggested by \citet{Jo83}. Fig.~\ref{fig:k5} illustrates this example; in fact, Fig.~\ref{fig:k5} can be treated as an alternative argument showing that the identity $x^2yx\bumpeq xyx^2$ fails in $\mathcal{K}_5$ in a way that complies with the geometric approach of the present paper.

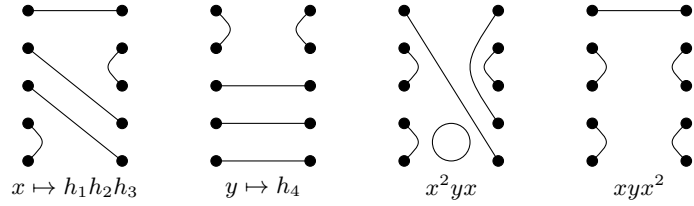
\begin{figure}[htb]
\begin{center}
\begin{tikzpicture}
[scale=0.5]
\draw (2.25,-0.7) node{$x\mapsto h_1h_2h_3$};
\foreach \x in {1,3.5} \foreach \y in {0,1,2,3,4} \filldraw (\x,\y) circle (4pt);
\draw (1,4) -- (3.5,4);
\draw (1,3) -- (3.5,1);
\draw (1,2) -- (3.5,0);
\draw (1,0) .. controls (1.5,0.5) .. (1,1);
\draw (3.5,2) .. controls (3,2.5) .. (3.5,3);
\draw (7.25,-0.7) node{$y\mapsto h_4$};
\foreach \x in {6,8.5} \foreach \y in {0,1,2,3,4} \filldraw (\x,\y) circle (4pt);
\draw (6,0) -- (8.5,0);
\draw (6,1) -- (8.5,1);
\draw (6,2) -- (8.5,2);
\draw (6,3) .. controls (6.5,3.5) .. (6,4);
\draw (8.5,3) .. controls (8,3.5) .. (8.5,4);
\draw (12.25,-0.7) node{$x^2yx$};
\foreach \x in {11,13.5} \foreach \y in {0,1,2,3,4} \filldraw (\x,\y) circle (4pt);
\draw (11,4) -- (13.5,0);
\draw (11,0) .. controls (11.5,0.5) .. (11,1);
\draw (11,2) .. controls (11.5,2.5) .. (11,3);
\draw (13.5,2) .. controls (13,2.5) .. (13.5,3);
\draw (13.5,1) .. controls (12.5,2.5) .. (13.5,4);
\draw (12.25,0.5) circle (.5);
\draw (17.25,-0.7) node{$xyx^2$};
\foreach \x in {16,18.5} \foreach \y in {0,1,2,3,4} \filldraw (\x,\y) circle (4pt);
\draw (16,1) .. controls (16.5,.5) .. (16,0);
\draw (16,3) .. controls (16.5,2.5) .. (16,2);
\draw (18.5,0) .. controls (18,0.5) .. (18.5,1);
\draw (18.5,2) .. controls (18,2.5) .. (18.5,3);
\draw (16,4) -- (18.5,4);
\end{tikzpicture}
\end{center}
\caption{The identity $x^2yx\bumpeq xyx^2$ fails in $\mathcal{K}_5$}\label{fig:k5}
\end{figure}

At the moment, we possess no characterization of the identities of the monoid $\mathcal{K}_n$ for any $n>4$, neither we know whether there are other pairs of equationally equivalent Kauffman monoids besides $\mathcal{K}_3$ and $\mathcal{K}_4$.

\end{document}